\documentclass[11pt]{article} 

\usepackage[utf8]{inputenc} 

\usepackage{geometry} 
\usepackage{graphicx} 

\usepackage{booktabs} 
\usepackage{array} 
\usepackage{paralist} 
\usepackage{verbatim} 
\usepackage{subfig} 
\usepackage{amsfonts}
\usepackage{amsthm}
\usepackage{amssymb}
\usepackage{amsmath}
\usepackage{mathrsfs}
\usepackage{hyperref}
\usepackage{amsthm}

\usepackage{fancyhdr} 
\usepackage{sectsty}
\allsectionsfont{\sffamily\mdseries\upshape} 

\usepackage[nottoc,notlof,notlot]{tocbibind} 
\usepackage[titles,subfigure]{tocloft} 

\newcommand{\ip}[2]{\langle #1,#2 \rangle}

\newcommand{\st}{~|~}

\newtheorem{lem}{Lemma}
\newtheorem{claim}{Claim}
\newtheorem{thm}{Theorem}

\newtheorem{cor}{Corollary}
\newcommand{\T}{\mathcal{T}}
\newcommand{\N}{\mathcal{N}}
\newcommand{\x}{\textbf{x}}

\newcommand{\supp}{\mathrm{supp}}

\makeatletter
\newtheorem*{rep@theorem}{\rep@title}
\newcommand{\newreptheorem}[2]{%
\newenvironment{rep#1}[1]{%
 \def\rep@title{#2 \ref{##1}}%
 \begin{rep@theorem}}%
 {\end{rep@theorem}}}
\makeatother

\newreptheorem{thm}{Theorem}
\newreptheorem{lem}{Lemma}
\newreptheorem{claim}{Claim}

\theoremstyle{definition}
\newtheorem{example}{Example}[section]
\newtheorem{define}{Definition}

\DeclareMathOperator*{\E}{\mathbb{E}}

\title{A Quantitative Local Limit Theorem for Triangles in Random Graphs}
\author{Ross Berkowitz
\thanks{Research supported in part by NSF grants CCF-1253886 and CCF-1540634, and the US-Israel Binational Science Foundation grant 2014359.}}

\begin{document}
\maketitle

\begin{abstract}
In this paper we prove a quantiative local limit theorem for the distribution of the number of triangles in the Erd\H{o}s-Renyi random graph $G(n,p)$, for a fixed $p\in (0,1)$.  This proof is an extension of the previous work of Gilmer and Kopparty, who
proved that the local limit theorem held asymptotically for triangles.  Our work gives bounds on the $\ell^1$ and $\ell^\infty$ distance of the triangle distribution from a suitable discrete  normal.  
\end{abstract}

\section{Introduction}
This paper is concerned with the distribution of the number of triangles appearing in an Erd\H{o}s-Renyi random graph $G(n,p)$ (a graph with $n$ vertices where each edge is present independently with probability $p$).  Recently, \cite{JustinTriangles} showed a local limit theorem in this context which says that the distribution of the number of triangles approaches the discrete normal.  Our main results show \textit{quantitative} bounds, both pointwise and global, on how far the distribution of the number of triangles
in a random graph can vary from a normal distribution.  In particular, if $\T$ is the random variable corresponding to the number of triangles in $G(n,p)$
we show that for all $k\in \mathbb{Z}$ and $\epsilon>0$,
$$\Pr[\T=k]=\frac{1}{\sqrt{2\pi}\sigma}\exp\left(-\frac{\left(k-\mu\right)^2}{2\sigma^2}\right)+O(n^{-2.5+\epsilon})$$
where $\mu=\E[\T]=p^3{n\choose 3}$ and $\sigma = Var(\T)$.
From this we are also able to obtain a quantitative bound on the $\ell^1$ distance of $\T$ from a suitable discrete normal:
$$\sum_{k\in \mathbb{N}}\left|\Pr(\T=k)-\frac{1}{\sqrt{2\pi}\sigma}\exp\left(-\frac{\left(k-p^3{n\choose 3}\right)^2}{2\sigma^2}\right)\right|=O(n^{-.5+\epsilon})$$

\subsection{History}

The study of subgraph counts dates back to the very beginning of random graphs, when Erd\H{o}s and Renyi proved in 1960 \cite{ErdosRenyi} that certain subgraph counts behaved in expected ways by using the second moment method.  In the 1980's there were several papers studying which subgraph counts obeyed a central limit theorem  (see \cite{Kar83, Kar84, NW88}).   For example, in this period a central limit theorem was shown for the triangle counting random variable $\T$, which stated that for any real numbers $a<b$
$$\Pr\left[\T\in[\mu+a\sigma_n, \mu+b\sigma_n]\right]=\frac{1}{\sqrt{2\pi}}\int_{a}^b e^{-t^2/2}dt+o(1)$$
  
This line of work eventually found a complete solution in the work of 
Ruci\'nski \cite{Ruc88} who gave a characterization for when subgraph counts obeyed a central limit theorem.
In 1989 there was progress made on showing central limit theorems with quantitative bounds in the work of Barbour, Karo\'{n}ski and Ruci\'nski \cite{BKR}.
Slightly afterwards Janson and Nowicki \cite{JansonNowicki} gave alternate arguments for central limit theorems using the
language of U-statistics using a good basis for functions on the probability space of graphs.  

If the edge probability $p\sim cn$ for some constant $c$, then Erd\H{o}s and Renyi \cite{ErdosRenyi} showed that the number of triangles in $G(n,p)$ converges to a Poisson distribution.  This result was a local limit theorem, as it estimated the pointwise probabilities $\Pr[\T=k]$ for $k$ constant.  Further, R\"{o}llin and Ross \cite{Ross} showed a local limit theorem when $p\sim cn^{\alpha}$ for $\alpha\in [-1,-\frac12]$.   In this regime they showed that the triangle counting distribution converges to a translated Poisson distribution (which is in turn close to a discrete Gaussian) in both the $\ell_\infty$ and total variation metrics.

In 2014, Gilmer and Kopparty \cite{JustinTriangles} proved a local limit theorem for triangle counts for $G(n,p)$  in the regime where $p$ is a fixed constant.  In particular they proved that
$$\Pr[\T=k]=\frac{1}{\sqrt{2\pi}\sigma}\exp\left(-\frac{\left(k-\mu\right)^2}{2\sigma_n^2}\right)\pm o(n^{-2})$$
It should be noted that this is largely a \textit{qualitative} result, as the main term has size $\Theta(n^{-2})$ while the error term is $o(n^{-2})$.  This type of result should also be contrasted with the central limit theorem given above.  This theorem gives an estimate for the probability of having \textit{exactly} $k$ triangles or differing from the expected number of triangles by exactly 17.  The central limit theorems estimate the probability of having a number
of triangles in an interval of length proportional to the standard deviation.

The proof in \cite{JustinTriangles} proceeded by using the characteristic function.  The main step there was to show that $|\varphi(t)-\varphi_n(t)|$ is small for $t\in [-\pi \sigma_n,~\pi\sigma_n]$, where $\varphi$ represents the characteristic function of the standard normal distribution, and $\varphi_n$ represents the characteristic function the triangle counting function $\T$.


\subsection{Our Results}

We improve the result of Gilmer and Kopparty by adding a quantitative estimate for the convergence of $\T$ to the normal.  We strengthen their bound to give explicit distance bounds.

\begin{thm} \label{Sup Main}
For any $k\in \mathbb{N}$ we have that
$$\Pr[\T=k]=\frac{1}{\sqrt{2\pi}\sigma_n}e^{-\frac{\left(k-p^3{n\choose 2}\right)^2}{2\sigma_n^2}}+O(n^{-2.5+\epsilon})$$
\end{thm}

For $k=\mu_n+O(\sigma_n)$ this shows that $\Pr[\T=k]$ is within a $(1+O(n^{-\frac12+\epsilon}))$ multiplicative factor of $\frac{1}{\sqrt{2\pi}\sigma}\exp\left(-\frac{\left(k-\mu\right)^2}{2\sigma_n^2}\right)$, while the best known previous bound could only show a factor of $(1+o(1))$.  A polynomial factor is also the best possible bound, as even the binomial distribution of ${n \choose 3}$ i.i.d. summands differs from the normal by a polynomial factor.
As a consequence of Theorem \ref{Sup Main} we also find a quantitative bound on the $\ell^1$ distance between $T$ and the normal.

\begin{thm}\label{L1 Main}
$$\sum_{t\in \mathbb{N}}\left|\Pr(\T=t)-\frac{1}{\sqrt{2\pi}\sigma}\exp\left(-\frac{\left(t-p^3{n\choose 3}\right)^2}{2\sigma^2}\right)\right|=O(n^{-0.5+\epsilon})$$
\end{thm}

The results in \cite{JustinTriangles} were not enough to imply $\ell^1$ distance bounds, so this is the first result of this kind for triangle counts.
Our arguments are based on giving better bounds on the characteristic function of $\T$.  The
 main improvments come from viewing the triangle counting function as a function over $\{0,1\}^{[n]\choose 2}$, and choosing a suitable basis for this space of functions.  This method is closely related to the method employed by Janson and Nowicki in \cite{JansonNowicki}.
In paricular, we will choose the $p$-biased Fourier basis over this space of functions given, with basis functions denoted $\chi_S$ where $S\subset {[n]\choose 2}$.  That is, each basis element is a function depending on some subset of the possible
edges in our graph.  The main mass of the triangle counting function will come from basis elements of the form $\chi_e$, where $e$ is some edge in ${[n]\choose 2}$.  In other words, we will find that
$\T$ is highly concentrated on its weight 1 Fourier coefficients.  This allows us to show that $\T$ may be reasonably well approximated as simply a linear function of the number of edges in the random graph.  Informally, this follows the intuition that
if one wanted to know how many triangles are in a fixed graph $G$, a reasonable estimator would be to simply ask how many edges are in the graph, and scale appropriately.

The actual estimation will be performed in two steps.  First we will normalize $\T$ to have mean 0 and variance 1 by defining $Z:=\frac{\T-\mu}{\sigma}$.  Then we will split $Z$ up into two pieces $Z=X+Y$, where $X$ carries the weight 1 fourier
terms which dominate $Z$, and $Y$ contains the higher order terms, which we will treat as error terms.  We then use as blunt a tool as the mean value theorem to estimate the characteristic function by saying 
$\E[e^{itZ}]=\E[e^{itX}+O(|tY|)]$.  Since $X$ is a sum of i.i.d. random variables and $Y$ is small, we will get that this converges to the characteristic function of the normal distribution when $t$ is small.

For slightly larger $t$ we adapt this method slightly, by first revealing some $k$-regular subgraph and then performing our estimates given this information.  This will shrink the size of $Y$ by a factor of $(k/n)^2$, but only shrink $X$ by a factor of $k/n$.  This gives us a better error term, at the cost of only slightly shrinking our main term.  For this part of the argument we cannot give an exact main term for $|\varphi_Z(t)|$, as we could in the first method.  However for $t$ large, because
the normal has very small characteristic function it suffices simply to show that $|\varphi_Z(t)|$ is very small as well.

\subsection{Organization of this Paper}
In section \ref{prelimsection} we set up our notation and introduce some facts which will be necessary for the later sections.  Section \ref{mainsection} contains the statements and proofs of our main results, modulo the main technical lemmas.  In section \ref{analysissection} we examine the decomposition of $\T$ with respect to the $p$-biased Fourier basis, and in section \ref{lemmasection} we exploit this decomposition to prove our main lemmas.  Finally in section \ref{graphstatsection} we extend these arguments to a more general setting to capture larger subgraph counts.

\section{Preliminaries and Notation} \label{prelimsection}
We will be working with a random variable which is defined as a graph function applied to an Erd\H{o}s-Renyi random graph $G(n,p)$.  We will be working in the regime where our probability $p$ is a fixed constant, and $n\to \infty$.
We will realize our probability space as drawing $\x\in \{0,1\}^{n\choose 2}$ where each coordinate of $\x$ is labelled by an edge $e\in {[n]\choose 2}$, and we have that for all edges, $\x_e$ is 0 with probability $1-p$ and 1 with probability $p$.  ${[n] \choose 2}$ refers equivalently to either the set of all pairs of distinct elements from $[n]$, or the set of possible edges of a graph with vertex set $[n]$.

Continuing our notation from the abstract, we use $\T:\{0,1\}^{n\choose 2}\to \mathbb{N}$ to denote the triangle counting function, which returns the number of triangles in the graph with edge set given by the indicator vector $\{0,1\}^{n\choose 2}$.  One might note that the random variable $\T$ depends on both the probability $p$, and the size of the vertex set $n$ in question.  We will often supress the dependence on $n$ and $p$, as we will be considering $p$ to be fixed
and our analysis will be done for a generic $n$, with limits only taken in the proof of the main theorem.

\subsection{$p$-Biased Fourier Basis}

To apply our analysis we use the $p$-biased Fourier basis for functions on this probability space.  We define this as follows.  For each edge $e\in {[n]\choose 2}$ we define $\chi_e:\{0,1\}^{n \choose 2}\to \mathbb{R}$ as follows: 
$$\chi_e:=\chi_e(\x):=\frac{\x_e-p}{\sqrt{p(1-p)}}=
\begin{cases}
-\sqrt{\frac{p}{1-p}}&\mbox{if }\x_e=0\\
\sqrt{\frac{1-p}{p}} &\mbox{if }\x_e=1
\end{cases}
$$

This is just the transform of the bernoulli random variable $\x_e$ so that it has mean 0 and variance 1.  Now for an arbitrary set $S\subset[n]$ we can define
$$\chi_S:=\chi_S(\x):=\prod_{e\in S}\chi_e$$
We note that if we take our inner product of two functions $f,g:\{0,1\}^{n\choose 2}\to \mathbb{R}$ to be defined by $\E[fg]$, then $\{\chi_S~|~S\subset [n]\}$ is an orthonormal basis  (See \cite{ODonnell} chapter 10  for more detail on this topic).

For any function $f:\{0,1\}^{{n\choose 2}}\to \mathbb{R}$, if we define the Fourier transform $\hat f:\{0,1\}^{n\choose 2}\to \mathbb{R}$  to be
$$\hat f(S):=\E[f(x)\chi_S(x)]$$
then by orthonormality we have that 
$$f(\x)=\sum_{S\subset {[n]\choose 2}} \hat f(S)\chi_S(\x)$$

\subsection{Probability Terminology and Notation}
In proving limit theorems, it is convenient to normalize the family of random variables to have mean 0 and variance 1, so throughout this chapter we will usually work with the related random variable $Z:\{0,1\}^{[n]\choose 2} \to \mathbb{R}$
$$Z(\x):=Z_n(\x):=\frac{\T-\mu}{\sigma}$$

We will frequently refer to the characteristic function of $Z$ as $\varphi_Z(t):=\E[e^{itZ}]$.  Most of the work will be focused on studying $\varphi_Z$, and showing it is close to $e^{-t^2/2}$.

We will also throughout the chapter label the variance of $\T$ as $\sigma^2:=\sigma_n^2:=\E[\T^2]-\E[\T]^2$.  A consequence of orthonormality gives us the following result, sometimes called Parseval's Theorem:

\begin{equation}\label{parseval}
\sigma^2:=\E[\T^2]-\E[\T]^2=\left(\sum_{S\subset {[n]\choose 2}} \hat \T(S)^2\right)-\hat \T(\varnothing)^2=\sum_{S\neq \varnothing} \hat \T(S)^2
\end{equation}

\subsection{Some Graph Notation}
Let $G$ be a graph with vertex set $[n]$ and edge set $E\subset {[n] \choose 2}$.  Given a triangle $\triangle$ with vertex set $\{v_1,v_2,v_3\}\subset [n]$ we will use the notation $e\in \triangle$ to denote that $e$ is an edge in the triangle $\triangle$ i.e. $e\in {\{v_1,v_2,v_3\} \choose 2}$.  Additionally we will occasionally identify a triangle $\triangle$ with its edge set.  That is, if we have $S\subset {[n]\choose 2}$ and we write $S=\triangle$, that means $S$ is the edge set of some triangle.

Also we will frequently need to refer to the case where $e_1$ and $e_2$ are two edges which are incident to a common vertex (i.e. $e_1=(v_1,v_2)$ and $e_2=(v_2,v_3)$).  We will denote this as $e_1\sim e_2$.

\subsection{Notation for function restrictions}

Often we will have a function $f:\{0,1\}^n\to \mathbb{R}$, and we will want to refer to the function obtained from $f$ by restricting some input coordinates to have certain values.  In particular assume that we have $H\subset [n]$ some fixed subset of input variables.  Then for $\beta\in \{0,1\}^{H^c}$ we will define $f_\beta:\{0,1\}^H\to \mathbb{R}$ by
$$f_\beta(\alpha)=f(\alpha,\beta)$$

\subsection{Ingredients for the Proof}
In this section we cite some useful results from the literature.  We will need the following Hypercontractivity result which bounds the probability that a low degree boolean function deviates from its mean.
\begin{thm}[\cite{ODonnell} Theorem 10.24] \label{Hypercontractivity}
Let $f:\{0,1\}^{n}\to \mathbb{R}$ be a polynomial of degree $k$, and $\lambda:=\min(p,1-p)$.  If $x\in \{0,1\}^n$ is chosen by setting each coordinate independently to be 1 with probability $p$ and 0 with probability $1-p$ then for any $t\ge \sqrt{2e/\lambda}^k$,
$$\Pr\left(|f(x)|\ge t\|f\|_2\right)\le \lambda^k\exp\left(-\frac{k}{2e}\lambda t^{\frac{2}{k}}\right)$$
\end{thm}

We will also use some of the existing bounds on the characteristic function of $\T$, which were derived in Gilmer-Kopparty.  We slightly modify their result to have a different choice of numbers, but the proof remains unchanged.
\begin{lem}[GK \cite{JustinTriangles} Theorem 5]\label{Justin}
Fix $\epsilon>0$.  If $\varphi_n(t)$ is the characteristic function of $Z=\frac{\T-p^3{n\choose 3}}{\sigma}$, then for $|t|\in [n^{.5+\epsilon},\pi\sigma_n]$ it holds that $|\varphi_n(t)|=O(|t|^{-50})$.
\end{lem}

We will frequently deal with Bernoulli random variables, and so the following bound on their characteristic function will be useful.
\begin{lem}\label{Bernoulli}
Let $X$ be the mean 0 variance 1 random variable taking the values
$$X:=\begin{cases}
-\sqrt{\frac{p}{1-p}}&\mbox{with probability  }1-p\\
\sqrt{\frac{1-p}{p}} &\mbox{with probability }p
\end{cases}
$$
Then for $|t|<\sqrt{p(1-p)}\pi$ we have that $|\E[e^{itX}]|<1-\frac{2 t^2}{\pi^2}$.
\end{lem}
\begin{proof}
Let $Y$ be the random variable taking the value $-1$ with probability $p$ and $1$ with probability $1-p$.  $Y$ has variance $4p(1-p)$, and $X=\frac{Y-\E[Y]}{2\sqrt{p(1-p)}}$.  Define $\tilde t:=\frac{t}{2\sqrt{p(1-p)}}$.  So we can compute that 
\begin{align*}|\E[e^{itX}]|^2&=\left|\E\left[e^{i\tilde tY}\right]\right|^2=|pe^{-i\tilde t}+(1-p)e^{i\tilde t}|^2=\|(\cos(\tilde t),~(1-2p)\sin(\tilde t))\|^2\\
&=1-4p(1-p)\sin^2(\tilde t)\le 1-\frac{16p(1-p)}{\pi^2}{\tilde t}^2\le 1-\frac{4t^2}{\pi^2}
\end{align*}
where we used the fact that $|\sin(\tilde )| \ge \frac{2|\tilde |}{\pi}$ for $|\tilde|\le \frac{\pi}{2}$.  Lastly noticing that $\sqrt{1-x}\le 1-\frac{x}{2}$  completes the proof.
\end{proof}

\section{Main Results} \label{mainsection}

Here we give the high level proof of our main results, deferring the proofs of the important lemmas to the next section.  First, we need the following standard theorem from probability.

\subsection{Local Limit Theorem for $\T$.}

\begin{thm}[Fourier Inversion Formula for Lattices]
Let $X$ be a random variable supported in $b+h\mathbb{Z}$, and let $\varphi(t)$ be the characteristic function of $X$.  Then for $x\in b+h\mathbb{Z}$
$$\mathbb{P}(X=x)=\frac{h}{2\pi}\int_{-\frac{\pi}{h}}^{\frac{\pi}{h}} e^{-itx}\varphi_X(t)dt$$
\end{thm}

As a consequence of this lemma we can turn characteristic function bounds for sequences of random variables into statements about their limiting distribution.
\begin{lem}\label{Pointwise Convergence}
Let $Y$ be the standard normal distribution which has density $\mathcal{N}(x)=\frac{1}{\sqrt{2\pi}}e^{-\frac{x^2}{2}}$ and characteristic function $\varphi(t)=e^{-\frac{t^2}{2}}$.   Let $X_n$ be a sequence of random variables
supported in the lattices $\mathcal{L}_n=b_n+h_n\mathbb{Z}$, then
$$|h_n\mathcal{N}(x)-\mathbb{P}(X_n=x)|\le h_n\left(\int_{-\frac{\pi}{h_n}}^{\frac{\pi}{h_n}}\left|\varphi(t)-\varphi_n(t)\right|dt+\frac{1}{\sqrt{2\pi}t}e^{-\frac{t^2}{2}}\right)$$
\end{lem}

\begin{proof}
By the general (that is, not the lattice version above) inversion principle for characteristic functions, we have $\N(x)=\frac{1}{2\pi}\int_{-\infty}^\infty e^{-itx}\varphi(t)dt$.
  By the above theorem we have
that $\mathbb{P}(X_n=x)=\frac{h_n}{2\pi}\int_{-\frac{\pi}{h_n}}^{\frac{\pi}{h_n}} e^{-itx}\varphi_n(t)dt$.  So we have that
\begin{align*}
\left|h_n\N(x)-\mathbb{P}(X_n=x)\right|&=\left|\frac{h_n}{2\pi}\int_{-\infty}^\infty e^{-itx}\varphi(t)dt-\frac{h_n}{2\pi}\int_{-\frac{\pi}{h_n}}^{\frac{\pi}{h_n}} e^{-itx}\varphi_n(t)dt\right|\\
&\le \left|\frac{h_n}{2\pi}\int_{-\frac{\pi}{h_n}}^{\frac{\pi}{h_n}} e^{-itx}\left(\varphi(t)-\varphi_n(t)\right)dt\right|+\left|\frac{h_n}{2\pi}\int_{|t|>\frac{\pi}{h_n}} e^{-itx}\varphi(t)dt\right|\\
&\le h_n\left(\int_{-\frac{\pi}{h_n}}^{\frac{\pi}{h_n}}\left|\varphi(t)-\varphi_n(t)\right|dt+\frac{1}{\sqrt{2\pi}t}e^{-\frac{t^2}{2}}\right)
\end{align*}
\end{proof}

The main calculation of this chapter is the following theorem, whose proof is given in Section 5.
\begin{thm}\label{MainTriangle}
Fix $\epsilon>0$.  Let $Z:=\frac{\T-p^3{n\choose 3}}{\sigma}$, and $\varphi_Z(t)$ be the characteristic function of $Z$.  Then 
$$\int_{-\pi \sigma_n}^{\pi \sigma_n}\left|\varphi_Z(t)-e^\frac{-t^2}{2}\right|=O_\epsilon(n^{-.5+\epsilon})$$
\end{thm}

We can now prove our main claim, Theorem \ref{Sup Main}, as it is elementarily equivalent to the following corollary.  
\begin{cor}\label{Sup Bound}
Let $\mathcal{L}_n:=\frac{1}{\sigma_n}(\mathbb{Z}-p^3{n\choose 3})$.  Then for any $x\in \mathcal{L}_n$ we have that
$$\left|\mathbb{P}(Z_n=x)-\frac{\N(x)}{\sigma_n}\right|=O_\epsilon\left(\frac{1}{n^{2.5-\epsilon}}\right)$$
\end{cor}
\begin{proof}
Apply Lemma \ref{Pointwise Convergence} to $Z$, combined with the estimate for the characteristic function of $Z$ given by Theorem \ref{MainTriangle}.
\end{proof}

\subsection{Bounds on the Statistical Distance of $\T$ from Normal}
We give a lemma which will allow us to turn the $L^\infty$ bounds we obtain into bounds on the statistical difference of $\T$ from the normal.
\begin{lem}\label{L1 Lemma}

Let $\N$ be the density of the standard normal and $\varphi(t)$ its characteristic function.  Let $X_n$ be a sequence of random variables supported in the lattice 
$\mathcal{L}_n:=b_n+h_n\mathbb{Z}$,
and with chf's $\varphi_n$.  Assume that the following hold:
\begin{enumerate}
\item $\sup_{x\in \mathcal{L}_n} |\Pr(X_n=x)-h_n\N(x)|<\delta_n h_n$
\item $\Pr(|X_n|>A)\le \epsilon_n$
\end{enumerate}
Then $\sum_{x\in \mathcal{L}_n} |\Pr(X_n=x)-\N(x)|\le 2A \delta_n+\epsilon_n+\frac{h_n}{\sqrt{2\pi}A}e^\frac{-A^2}{2}$.
\end{lem}
\begin{proof}
We directly compute that: 
\begin{align*}
\sum_{x\in \mathcal{L}_n} |\Pr(X_n=x)-h_n\N(x)|&\le \sum_{\substack{x\in \mathcal{L}_n\\|x|<A}} |\Pr(X_n=x)-h_n\N(x)| + \sum_{\substack{x\in \mathcal{L}_n\\|x|\ge A}} |\Pr(X_n=x)-h_n\N(x)|\\
&\le \sum_{\substack{x\in \mathcal{L}_n\\|x|<A}} |\Pr(X_n=x)-h_n\N(x)|+\Pr(X_n\ge A)+h_n\int_{|x|>A-1} \N(x)dx\\
&\le \frac{2A}{h_n} \delta_n h_n+\epsilon_n+\frac{h_n}{\sqrt{2\pi}A}e^\frac{-A^2}{2}
\end{align*}
\end{proof}

We can now use this to give a proof that the statistical distance between triangle counts and discrete normal variable is asymptotically small.  We will pick $A:=\log^2(n)$.  By an application
of hypercontractivity (Theorem \ref{Hypercontractivity}) we find that 
$$\Pr(|Z_n|>\log^2(n))\le e^{-\Omega_p(\log^2(n))}=n^{-\Omega_p(\log(n))}=o\left(n^{-.5}\right)$$
This bounds the $\epsilon_n$ term in the above theorem.
We also have from Corollary \ref{Sup Bound} that $\sup_{x\in \mathcal{L}_n} |\Pr(X_n=x)-h_n\N(x)|=O_\epsilon(n^{-2.5+\epsilon})$.  Combining this with
the calculation that $\sigma_n=\Theta(n^2)$ we obtain the following corollary, which is equivalent to Theorem \ref{L1 Main}:

\begin{cor}\label{L1 Bound}
Fix $\epsilon>0$.  Let $\mathcal{L}_n:=\frac{1}{\sigma_n}(\mathbb{Z}-p^3{n\choose 3})$.  Then
$$\sum_{x\in \mathcal{L}_n}\left|\Pr(Z=x)-\frac{1}{\sigma}\N(x)\right|=O_\epsilon(n^{-.5+\epsilon})$$
\end{cor}
\begin{proof}
 In the above Lemma for $X_n=\T_n$ we have that $h_n=\sigma_n$.  We may take $\delta_n:=n^{-.5+\frac{\epsilon}{2}}$ by Corollary \ref{Sup Bound}, and we may fix $A=\log^2(n)$ as above.  Then as argued above $\epsilon_n=O(n^{-.5})$ while $e^{-\frac{A^2}{2}}$ is miniscule.  Plugging these choices into the bound given by Lemma \ref{L1 Lemma} gives the desired estimate.
\end{proof}

\section{Properties of the Triangle Counting Function}\label{analysissection}
In this section we express the triangle counting function in the $p$-biased Fourier basis, and compute some basic properties.

Given a particular triangle $\triangle$ with vertex set $v_1,v_2,v_3$, we will use the notation $e\in \triangle$ to denote that $e$ is an edge in the given triangle $\triangle$.  The indicator function of this triangle's presence  given the graph with edge indicator vector $\x\in \{0,1\}^{n\choose 2}$ is given by
\begin{align*}
1_\triangle(\x)&=\prod_{e\in \triangle}\x_e= \prod_{e\in \triangle}\left(\sqrt{p(1-p)}\chi_e(\x)+p\right)\\\
&=p^3+p^2\sqrt{p(1-p)}\sum_{e\in \triangle}\chi_e+p^2(1-p)\sum_{e_1\neq e_2\in \triangle}\chi_{\{e_1,e_2\}}+(p(1-p))^\frac 3 2 \chi_{\{e_1,e_2,e_3\}}
\end{align*}

Given two edges every edge appears in $n-2$ triangles, and each pair of edges appear in exactly 1 triangle iff they are incident to a common vertex (an event which we denote by $e_1\sim e_2$) we find by summing over all possible triangles that

$$\T=p^3{n \choose 3}+(n-2)\sum_{e\in {[n]\choose 2}} p^2\sqrt{(p)(1-p)}\chi_e+\sum_{e_1\sim e_2}p^2(1-p)\chi_{\{e_1,e_2\}}+\sum_{\triangle}p^\frac32(1-p)^\frac32 \chi_\triangle$$

Restated we have found the Fourier Transform of $\T$ and it has the form
\begin{equation}\label{FourierTransform}
\hat \T(S)=\begin{cases}
p^3{n\choose 3}&\mbox{if } S=\varnothing\\
 (n-2)p^2\sqrt{(p)(1-p)}&\mbox{if }|S|=1\\
p^2(1-p)&\mbox{if } S=\{e_1,e_2\},~e_1\sim e_2\\
p^\frac32(1-p)^\frac32 &\mbox{if }S=\triangle\\
0&\mbox{else}
\end{cases}
\end{equation}

We compute the variance of $\T$ using the orthonormality of our basis (or Parseval) to be
\begin{equation} \label{Std}
\begin{aligned}
\sigma^2:&=\E[\T^2]-\E[\T]^2=\sum_{\substack{S\subset {[n]\choose 2}\\S\neq \varnothing}} \hat{T}^2(S)\\
&=\sum_{e\in {[n]\choose 2}} \left((n-2)p^2\sqrt{(p)(1-p)}\right)^2+\sum_{e_1\sim e_2}\left(p^2(1-p)\right)^2+\sum_{\triangle}\left(p^\frac32(1-p)^\frac32\right)^2\\
&={n\choose 2}(n-2)^2p^5(1-p)+3{n\choose 3} p^4(1-p)^2+{n\choose 3}p^3(1-p)^3\\
&=\Theta(n^4)
\end{aligned}
\end{equation}

It should be noted that asymptotically we have $\sigma\sim \frac{p^{5/2}(1-p)^{1/2}}{2}n^2$.  Also it is significant that the main term in the above expansion of $\sigma^2$ comes entirely from terms of the form $\chi_e$, for a singleton set containing one
edge $e$.  This shows that $\T$ has Fourier spectrum highly concentrated on degree 1.  In particular, if we define $W^1:=\sum_e \hat \T^2(e)$ then $\sigma^2=W^1(1+O(\frac1n))$.

Recall that we defined $Z:=\frac{\T-\mu}{\sigma}=\frac{\T-p^3{n\choose 3}}{\sigma}$.   By construction $Z$ has mean 0 and variance 1.
The fourier decomposition of $Z$ is just a normalized version of $\T$.
In particular $\hat Z(S)=\frac{\hat \T(S)}{\sigma}$ if $S\neq \varnothing$, and $\hat Z(\varnothing)=0$.

\section{Estimating the Characteristic Function of $Z$}\label{lemmasection}
\subsection{Main Results of the Section}
In this section  we prove the following bound
\begin{thm}\label{MainTriangle}
Let $Z:=\frac{\T-p^3{n\choose 3}}{\sigma}$, and $\varphi_Z(t)$ be the characteristic function of $Z$.  Then for any $\epsilon>0$
$$\int_{-\pi \sigma_n}^{\pi \sigma_n}\left|\varphi_Z(t)-e^\frac{-t^2}{2}\right|=O_\epsilon(n^{.5-\epsilon})$$
\end{thm}
The work is done over 3 sections, each corresponding to different sizes of $t$.
In Section \ref{smallt} we prove the following bound which, while true for all $t$, is most useful for smaller values of $t$
\begin{lem}\label{smalltbound}
\begin{align*}
\left|\varphi_Z(t)-e^{-\frac{t^2}{2}}\right|=O\left(\frac{t^3e^{-\frac{t^2}{3}}}{n}+\frac{t}{\sqrt{n}}\right)
\end{align*}
\end{lem}
Subsequently in Section \ref{midt} we prove the result for ``mid-sized'' $t$ that
\begin{lem}\label{midtbound}
Fix $0<\epsilon<1$.  Then
$$\int_{n^{\epsilon}}^{n^{\frac{1+\epsilon}{2}}} |\varphi_Z(t)|dt \le O_{\epsilon}(n^{-.5+\epsilon})$$
\end{lem}

Lastly for $|t|\ge {n^{\frac{1+\epsilon}{2}}}$ we simply cite Lemma \ref{Justin}.
Combining all these results immediately gives Thoerem \ref{MainTriangle}.  For completeness we give the proof.
\begin{proof}[Proof of Theorem \ref{MainTriangle}]
\begin{align*}
\int_{-\pi \sigma_n}^{\pi \sigma_n}\left|\varphi_Z(t)-e^\frac{-t^2}{2}\right|&=\int_{|t|<n^{0.05}}\left|\varphi_Z(t)-e^\frac{-t^2}{2}\right|+\int_{n^{0.05}<|t|<n^{.5+\frac{\epsilon}{10}}}\left|\varphi_Z(t)-e^\frac{-t^2}{2}\right|\\
&+
\int_{n^{.5+\frac{\epsilon}{10}}<|t|<\pi\sigma_n}\left|\varphi_Z(t)-e^\frac{-t^2}{2}\right|\\
&\le \int_{|t|<n^{\epsilon}}O\left(\frac{t^3e^{-\frac{t^2}{3}}}{n}+\frac{t}{\sqrt{n}}\right)dt+O_{p,\epsilon}(n^{-.5+\epsilon})+O(n^{-50})+2\left|\int_{n^{\epsilon}}^{\pi \sigma_n} {e^{-\frac{t^2}{2}}}dt\right|\\
&=O_\epsilon(n^{-0.5+2\epsilon})
\end{align*}

\end{proof}

\subsection{Bounds for small $t$}\label{smallt}
In this section we prove the following result
\begin{replem}{smalltbound}
\begin{align*}
\left|\varphi_Z(t)-e^{-\frac{t^2}{2}}\right|=O\left(\frac{t^3e^{-\frac{t^2}{3}}}{n}+\frac{t}{\sqrt{n}}\right)
\end{align*}
\end{replem}
This shows that the characteristic function of $Z$ is very close to that of $N(0,1)$ for small $t$.  In that regard this result is a generalization of a central limit theorem for $T$, which is equivalent to the pointwise convergence of $\varphi_Z(t)$ to $e^{-t^2/2}$.
\begin{proof}
We can decompose $Z$ into two parts, the dominant weight one part $X$, and a smaller term corresponding to fourier coefficients of weight $\ge 2$.  In particular let $Q:=\sqrt{\frac{1}{{n\choose 2}}}$.  Then we define
\begin{align*}
X:=\sum_{e\in {[n]\choose 2}} Q \chi_e&&Y:=\sum_{e\in {[n]\choose 2}} (\hat Z(e)-Q)\chi_e+\sum_{|S|\ge 2} \hat Z(S) \chi_S
\end{align*}

First we examine $X$.  It is the mean 0 variance 1 sum of independent  random variables, and so by Berry-Esseen (see Petrov \cite{Petrov}, Chapter V lemma 1) we know that if 
$$L_n:={n\choose 2}\E[|Q\chi_e|^3]=\frac{p^2+(1-p)^2}{\sqrt{{n\choose 2}p(1-p)}}=\Theta_p\left(1/n\right)$$
then for $t\le \frac{1}{4L_n}$ we have that
\begin{equation}\label{eqBerryEsseen}
\left|\E[e^{itX}]-e^{-\frac{t^2}{2}}\right| \le 16L_n|t|^3e^{\frac{-t^2}{3}}
\end{equation}
%
%
%
%

Now we turn our attention to $Y$.  $Y$ is best thought of as an error term.  It is where the dependence of our random variable $Z$ lives, and it will be always very small.  In particular, using Cauchy-Schwarz and the orthogonality
of our basis we obtain
$$\E|Y|\le \sqrt{\E |Y|^2}=var(Y)=\sum_e (\hat Z(e)-Q)^2+\sum_{|S|\ge 2} \hat Z^2(S)$$

We know from prior calculations that 
$$\sum_{|S|\ge 2} \hat Z^2(S)=\frac{3{n\choose 3} p^4(1-p)^2+{n\choose 3}p^3(1-p)^3}{\sigma^2}=O\left(\frac1n\right)$$
Further we can estimate
$${n\choose 2}\sigma^2\hat Z^2(e)-\sigma^2= {n\choose 2}\hat T^2(e)-\sigma^2=O(n^3)\implies \hat Z^2(e)-\frac{1}{{n\choose 2}}=O(n^{-3})$$
Therefore using the fact that $(x-y)=(x^2-y^2)/(x+y)$ coupled with the observation that $\hat Z(e)+Q=\Theta\left(\frac{1}{n}\right)$, we find that
$$|\hat Z(e)-Q|\le \left|\frac{ \hat Z^2(e)-\frac{1}{{n\choose 2}}}{\hat Z(e)+Q}\right|=O\left(\frac{1}{n^2}\right)$$

So as a result we can conclude that $var(Y)=O(1/n)$ and so $\E[|Y|]=O(\frac{1}{\sqrt n})$.  Now we are ready for our characteristic function bound for $Z$.  If $|t|\le \frac{1}{4L_n}=\Theta_p(n)$ then combining the above with equation \ref{eqBerryEsseen} yields.

\begin{align*}
\left|\varphi_Z(t)-e^{-\frac{t^2}{2}}\right|&=\left|\E\left[e^{itZ}\right]-e^{-\frac{t^2}{2}}\right|=\left|\E\left[e^{it(X+Y)}\right]-e^{-\frac{t^2}{2}}\right|\le \left|\E\left[e^{itX}\right]-e^{-\frac{t^2}{2}}\right|+\left|\E\left[e^{itX+Y}\right]-\E e^{itX}\right|\\
&\le 16L_n|t|^3e^{\frac{-t^2}{3}}+\E|tY|=O\left(\frac{t^3e^{-\frac{t^2}{3}}}{n}+\frac{t}{\sqrt{n}}\right)
\end{align*}
The last inequality comes from simply applying the mean value theorem to the function $e^{itx}$.  The first term in the error is dominated for any choice of $t$, and so we can simplify the error to $|\varphi_Z(t)-e^{-t^2/2}|=O(tn^{-1/2})$.
\end{proof}

\subsection{Bounds for slightly larger $t$}\label{midt}
Here we perform the same operation as above, except we first reveal a fraction of the edges.  The intuition behind this is that revealing a $q$ fraction of the edges will reduce the number of edge variables over which we take our expectation by $q$,
but it will reduce the influence of larger sets by even more, namely by $q^{|S|}\ge q^2$.  Thus in the above estimate when we decompose $Z$ into $X+Y$ we will find that $Y$ will be significantly smaller, allowing us to get a better estimate.

For any natural number $k$, we can take $H$ to be a $k$-regular bipartite graph on $n$ vertices.  Then it makes sense to talk about the restriction of $Z$ to the variables in $H$.  That is we are \textit{revealing} the edges in $H^c$ to be some vector $\beta\in \{0,1\}^{H^c}$, and consider the function $Z_\beta:\{0,1\}^H\to \mathbb{R}$ given by $Z_\beta(\alpha)=Z(\alpha,\beta)$.
First we note that by the law of total probability we have that
$$\E[e^{itZ}]=\E_{\beta\in\{0,1\}^{H^c}}\E_{\alpha \in \{0,1\}^{H}}[e^{itZ_\beta(\alpha)}]$$

So now we turn our attention to examining the form $Z_\beta$ takes for a typical restriction $\beta$.  First let us consider what happens to a generic basis function $\chi_S$ (
For a general consideration of how restriction interacts with fourier bases (particularly in the case of $p=\frac12$)  see \cite{ODonnell} Chapter 3.3 ).
  If we split $S$  as $S=S_H\cup S_{H^c}$ where $S_H\subset H$ and $S_H^c\subset H^c$ then
$$(\chi_S)_\beta(\x)=\chi_{S_{H^c}}(\beta)\chi_{S_H}(\x)$$
So we can use this to compute the Fourier transform of $Z_\beta:\{0,1\}^H\to \mathbb{R}$.  For an arbitrary $S\subset H$ we will have that
\begin{equation}\label{Coefficients}
\widehat{Z_\beta}(S)=\sum_{T\subset H^c} \chi_{T}(\beta)\hat Z(S\cup T)
\end{equation}

If we fix $S$, and think of $\beta$ as an input, then $\widehat{Z_\beta}(S)$ can be viewed as a function of $\beta$, $\widehat{Z_\beta}(S):\{0,1\}^{H^c}\to \mathbb{R}$.

\begin{claim}\label{GoodEvent}
Let $A$ be the event (over the space of revelations $\beta \in \{0,1\}^{H^c}$) that for \emph{every} edge $e\in H$ we have that
$$|\widehat{Z_\beta}(e)-\hat{Z}(e)|<\frac{\sqrt{3}n^{0.6}}{\sigma}$$
Let $\lambda:=\min(p,1-p)$.  Then $\Pr(A)\ge 1-n^2\lambda^2e^{-\lambda n^{.01}}$.
\end{claim}

\begin{claim}\label{GoodBound}
Assume $\beta\in A$.  Then for $t\le \sigma \pi\sqrt{p(1-p)}/2n=\Theta_p(n)$
$$|\E_{\alpha \subset H}[e^{itZ_\beta(\alpha)}]|\le \exp\left(-\frac{kt^2n^3}{4\pi^2 \sigma^2}\right)+\frac{4|t|n{k \choose 2}}{\sigma^2}$$
\end{claim}

Assuming these two claims we can prove the main result for this subsection.

\begin{replem}{midtbound}
Fix $0<\epsilon<1$.  Then
$$\int_{n^{\epsilon}}^{n^{\frac{1+\epsilon}{2}}} |\varphi_Z(t)|dt \le O_{\epsilon}(n^{-.5+\epsilon})$$
\end{replem}

\begin{proof}
Let $A$, as in Claim \ref{GoodEvent}, be the event that for all $e\in H$ we have that $|\widehat{Z_\beta}(e)-\hat{Z}(e)|<\frac{\sqrt{3}n^{0.6}}{\sigma}$.
We can break up $\{0,1\}^{H^c}$ into $A$ and $A^c$ and estimate

\begin{align*}
|\varphi_Z(t)|&:=\E_{(\alpha,\beta)\in 2^{n\choose 2}}[e^{itZ(\alpha,\beta)}]\le \E_{\beta\subset H^c}|\E_{\alpha\subset H}[e^{itZ_\beta(\alpha)}]|\le \Pr(A)+ (1-\Pr(A))\E_{\beta \in A^c}\left|\E_{\alpha}[e^{itZ_\beta}]\right|
\end{align*}
Now combining Claims \ref{GoodEvent} and \ref{GoodBound} we find that

\begin{align*}
\Pr(A^c)+ (Pr(A))\E_{\beta \in A}\left|\E_{\alpha}[e^{itZ_\beta}]\right|&\le  \lambda^2n^2e^{-\frac{\lambda}{e} n^{0.1}}+ \exp\left(-\frac{kt^2n^3}{4 \pi^2\sigma^2}\right)+\frac{2k|t|\sqrt{n}}{\sigma}
\end{align*}

We may choose $k$ to be an integer of size $n\lceil|t|^{-2+\epsilon}\rceil$ (which may be done for $0<|t|\le n^{\frac{1+\epsilon}{2}}$).  Recalling that $\sigma =\Theta(n^2)$ we find that
$$|\varphi_Z(t)|=O\left(n^2e^{-\frac{\lambda}{e} n^{0.1}}+\exp\left(-{\frac{-t^{\epsilon}n^4}{4\pi^2\sigma^2}}\right)+\frac{1}{|t|^{1-\epsilon}\sqrt{n}}\right)$$

Using this  we may make the following estimate
$$\int_{n^{\epsilon}}^{n^{\frac{1+\epsilon}{2}}} |\varphi_Z(t)|dt\le  O\left(n^{2+\frac{1+\epsilon}{2}}e^{-\frac{\lambda}{e} n^{0.1}}+n^{\frac{1+\epsilon}{2}}\exp\left(-n^{\epsilon}\right)+\left[ \frac{1}{\epsilon}t^{\epsilon}n^{-.5}\right]_{n^{\epsilon}}^{n^{\frac{1+\epsilon}{2}}}\right)=O_{\epsilon}\left( n^{-.5+\epsilon}\right)$$
\end{proof}

\subsubsection{Proof Of Claims \ref{GoodEvent} and \ref{GoodBound}}

\begin{repclaim}{GoodEvent}
Let $A$ be the event (over the space of revelations $\beta \in \{0,1\}^{H^c}$) that for \emph{every} edge $e\in H$ we have that
$$|\widehat{Z_\beta}(e)-\hat{Z}(e)|<\frac{\sqrt{3}n^{0.6}}{\sigma}$$
Let $\lambda:=\min(p,1-p)$.  Then $\Pr(A)\ge 1-n^2\lambda^2e^{-\lambda n^{.01}}$.
\end{repclaim}

We prove Claim \ref{GoodEvent} by noting that the formula for $\widehat{Z_\beta}(S)$ (a coefficient in the polynomial $Z_\beta$) is \emph{itself} a low degree polynomial, and therefore may be shown to have tight concentration by Theorem \ref{Hypercontractivity}.

\begin{proof}[Proof Of Claim \ref{GoodEvent}]

Recall equation \ref{Coefficients} which states that
$$\widehat{Z_\beta}(e)=\sum_{T\subset H^c} \hat Z(e\cup T)\chi_{T}(\beta)$$

$\widehat{Z_\beta}(e):\{0,1\}^{H^c}\to\mathbb{R}$ is a polynomial (in the functions $\chi_e$), and we can began by estimating its coefficients.
First we see that 
$$\E[\widehat{Z_\beta}(e)] =\widehat{\widehat{Z}_\beta(e)}(\varnothing)=\hat{Z}(e)$$

Also for any $T\subset \{0,1\}^{H^c}$ we know that $\hat{Z}(e\cup T)\neq 0$ iff $e$ and $T$ are in a common triangle.  There are at most $3(n-2)$ choices of $T$ (corresponding to completing the $n-2$ triangles containing the edge $e$).  Therefore Combining this with the fact that $\hat Z(S')\le \sigma^{-1}$ for all sets of size $|S'|\ge 2$ we find that
$$var_\beta( \widehat{Z_\beta}(e))=\sum_{\substack{T\subset H^c\\T\neq \varnothing}} \hat{Z}(e\cup T)^2\le \frac{3(n-2)}{\sigma^2}$$

Since $\widehat{Z_\beta}(e)$ has degree 2, an application of Theorem \ref{Hypercontractivity} gives us that for any $e\in H$ if we set $\lambda=\min(p,1-p)$ then
$$\Pr\left[\left|\widehat{Z_\beta}(e)-\hat Z(e)\right|\ge \frac{\sqrt{3}n^{0.6}}{\sigma}\right]<\lambda^2\exp\left({-\frac{\lambda n^{0.1}}{e}}\right)$$
Applying a union bound over all edges in $H$ completes the proof.
\end{proof}

Claim \ref{GoodBound} is concerned with estimating $|\E[e^{itZ_\beta}]|$, given that $\beta$ is a typical, well behaved revelation.  When $\beta$ is well behaved $Z_\beta$ will be dominated by a sum of independent monomials, and so the proof proceeds in a manner very similar to the arguments in Section \ref{smallt}.

\begin{repclaim}{GoodBound}
Recall $A$ is the event in  $2^{H^c}$ such that for all edges $e\in H$ we have $\left|\widehat{Z_\beta}(e)-\hat Z(e)\right|\le \sqrt{3}n^{.6}$ (that is, the set of all revelations of the edges of $H^c$ which are well behaved).

Assume $\beta\in A$.  Then for $t\le \sigma \pi\sqrt{p(1-p)}/2n=\Theta_p(n)$
$$|\E_{\alpha \subset H}[e^{itZ_\beta(\alpha)}]|\le \exp\left(-\frac{kt^2n^3}{4\pi^2 \sigma^2}\right)+\frac{2k|t|\sqrt{n}}{\sigma}$$
\end{repclaim}
\begin{proof}[Proof of Claim 2]
Assume that $\beta\in A$.  Let $X$ and $Y$ be
\begin{align*}
X:=\sum_{e\in {[n]\choose 2}} \widehat{Z_\beta}(e) \chi_e&&Y:=\sum_{|S|\ge 2} \widehat{Z_\beta}(S) \chi_S
\end{align*}
then $Z_\beta=X+Y$, and we will be able to obtain bounds similar to our previous ones.  In particular $X$ is the sum of indpendent random variables so if for each $e$ we define $Q_e:=\widehat{Z_\beta}(e)$ then
we will have because of our assumptions that $\frac{n}{2\sigma} \le \hat{Z}(e)-\frac{\sqrt{3}n^{.6}}{\sigma}\le Q_e\le \frac{2n}{\sigma}$

because $\widehat{Z_\beta}(e)\le \frac{2n}{\sigma}$ and $t\le \frac{\sigma\pi\sqrt{p(1-p)}}{2n}$ we can use Lemma $\ref{Bernoulli}$ to show that
$$|\E[e^{it\widehat{Z_\beta}(e)\chi_e}]|\le 1-\frac{t^2n^2}{2\pi^2\sigma^2}\le \exp\left(-\frac{t^2n^2}{2\pi^2\sigma^2}\right)$$ 

So now we find that
\begin{align*}
\E[e^{itX}]=\prod_{e\in H}\E[\exp\left(it\widehat{Z_\beta}(e)\chi_e\right)]\le \exp\left(-\sum_{e\in H}(t\widehat{Z_\beta(e)})^2\right)\le \exp\left(\sum_{e\in H}-\frac{t^2n^2}{\pi^2\sigma^2}\right)=\exp\left(-\frac{kt^2n^3}{4\pi^2\sigma^2}\right)
\end{align*}

Now we turn our attention to $Y$.  If $|S|=2$ with $S=\{e_1,e_2\}$ then $\hat{Z}(S)$ is 0 unless $e_1\sim e_2$, and therefore $e_1,e_2$ lie in a common triangle $\triangle=\{e_1,e_2,e_3\}$.  However this is the only triangle containing $S$, and so we can quickly compute using equation \ref{Coefficients}, and the fact for $|S|\ge 2$ we have $|\hat{Z}(S)|\le \frac{1}{\sigma}$ (see equation \ref{FourierTransform} and normalize to obtain $Z$) that
$$\widehat{Z_\beta}(S)=\sum_{T\subset H^c} \chi_{T}(\beta)\hat Z(S\cup T)=  \chi_{\varnothing}(\beta)\hat{Z}(S)+\chi_{e_3}(\beta)\hat{Z}(\triangle)\le \frac{2}{\sigma}$$
So we can compute, again using Cauchy Schwartz and the fact that $H$ is $k$-regular that
$$\E[|Y|]^2\le \E[|Y|^2]=\sum_{\substack{e_1\sim e_2\\e_1,e_2\in H}} \widehat{Z_\beta}^2(S)\le n{k \choose 2} \frac{4}{\sigma^2}$$
Combining this information, we compute that

\begin{align*}
\left|\E_{\alpha \in 2^{H}} [e^{itZ_\beta(\alpha)}]\right|&=\left|\E_{\alpha}[e^{it(X+Y)}]\right|\le \left|\E[e^{itX}+|tY|]\right|\\
&\le \exp\left(-\frac{kt^2n^3}{2\pi^2\sigma^2}\right)+|t|\sqrt{{k \choose 2}{n} \frac{(2)^2}{\sigma^2}}
\end{align*}
\end{proof}
\section{General Subgraph Counts in $G(n,p)$}\label{graphstatsection}
In this section we take the arguments we have used so far in this chapter and extend them to counting subgraphs other than triangles.  We will be able to give good characteristic function bounds
for the corresponding random variables, however these results as of yet do not yield any local limit theorems for any graphs on more than 3 vertices.  We will, however, be able to give a new proof of quantitative central limit theorems
for subgraph counts in $G(n,p)$.  Section \ref{subgraphsetup} will introduce necessary notation and definitions.   Section \ref{subgraphmains} will contain the main results of this section.  The remaining sections will cover the properties of graph
statitics and then the proofs of the theorems.
\subsection{Definitions and Graph Statistics}\label{subgraphsetup}
Falling factorials will frequently appear in our analysis, and we will use the following notation:
\begin{define}
Let $n,k\in \mathbb{N}$.  We define $(n)\downarrow_k:=\prod_{i=0}^{k-1}(n-i)$.  For the case $k=0$ we set $(n)\downarrow_{0}=1$.
\end{define}
\noindent Throughout this section we will be working with functions defined on graphs.  To capture subgraph counts we will need two graphs:  our random graph $G$ on a large growing vertex set of size $n$, and a second graph $\Gamma$ on vertex sets of a fixed size $k$ that will define the subgraphs we are interested in counting.
\begin{define} Let $\mathcal{S}_G:=\mathcal{S}_G(n,k)$ denote the set of all labeled (with vertices distinguishable by their origin in $[n]$, and given a labeling in $[k]$) induced subgraphs of the graph $G$ with $k$ vertices.  It will also be useful to denote this as the set of injections of $\psi:[k] \hookrightarrow [n]$, with the map extended to edges in the obvious way.
\end{define}
Let's denote the edge set in the big graph to be $E={[n] \choose 2}$ and the edge set in the small graph to be $D:={[k]\choose 2}$.
Here we will give a standard notation to a slight generalization of subgraph counts, which we will call graph statistics, and the rest of this section will be concerned with analyzing such functions
\begin{define}\label{graphstat}
Fix a graph function $f:2^{[k]\choose 2}\to\mathbb{R}$.  For any $n\ge k$ we can define the \emph{graph statistic} $F_f:2^{[n]\choose 2}$ (typically denoted simply as $F$) for $f$ to be
$$
F(G):=F_f(G):=\sum_{\Gamma\in \mathcal{S}_G} f(\Gamma)
$$
\end{define}
A function $F(G)$ defined this way sums $f$ as applied to all ordered subgraphs of size $k$ in $G$.  In particular if $f$ is the indicator of a fixed graph $H$ (induced or otherwise), then the graph statistic $F(G)$ counts the number of copies of this graph inside $G$.
To help our study of the properties of $F$, it will be useful to have some notation aggregating information about the base function $f$.
\begin{define}\label{hnotation}
For a set $T\subset E$ let
$$h_T:=\sum_{\phi:\supp(T) \hookrightarrow D} \hat f(\psi(T))$$
where the summation is over all injections of $\supp(T)$ into $D={[k]\choose 2}$.
\end{define}
Note the arrow here is {reversed} from the maps in the definition of $\mathcal{S}_G$.   Also, $h_T$ depends only on the isomorphism class of $T$, and importantly does not change as the parameter $n$ changes.  It will also be handy to define the largest such coefficient to be
\begin{equation}\label{suph}
h_*:=\max_{T} |h_T|
\end{equation}
When analyzing the low weight spectral concentration of $F$, a better measure for estimating $\hat{F}(S)$ than simply $|S|$, will be the number of vertices incident to edges in $S$.  We call this set of vertices the support of $S$.
\begin{define}
 Given a set of edges $S$, define  $\supp(S):=\cup_{e\in S}e$, the set of all vertices incident to edges in $S$.
\end{define}

Our main theorems in the next section will be aimed at bounds on the characteristics of subgraph counting random variables.  However, our arguments will work in the slightly more general setting of graph statistics which are \emph{edge dominated}.
\begin{define}
If $f$ has the property that $h_e=\sum_{e\in {[k]\choose 2}} \hat f(e)\neq 0$, then we say that $F$ is \emph{edge dominated}.
\end{define}
 A few examples to illustrate these definitions are in order.
\begin{example}
Consider $|\Gamma|=3$, and so $f:2^{[3]\choose 2}\to\mathbb{R}$ defined by 
$$f(\x)=\x_{12}\x_{23}=\left(\sqrt{p(1-p)}\chi_{(12)}(\x)+p\right)\left(\sqrt{p(1-p)}\chi_{(23)}(\x)+p\right)$$
Then $f$ is the indicator of whether the input graph $\Gamma$ contains the length 2 path from 1 to 3, but puts no condition on the edge
between vertices 2 and 3.  $F_f$ will count all ordered paths of length 2 in the graph and will be edge dominated for any $p$ (as can be seen by expanding out the above product).
\end{example}

\begin{example}
Again take $|\Gamma|=3$, and so $f:2^{[3]\choose 2}\to\mathbb{R}$ defined by 
$$f(\x)=\x_{12}\x_{23}(1-x_{13})=\left(\sqrt{p(1-p)}\chi_{(12)}(\x)+p\right)\left(\sqrt{p(1-p)}\chi_{(23)}(\x)+p\right)\left(\sqrt{p(1-p)}\chi_{(13)}(\x)+p-1\right)$$
Then $f$ is the indicator of whether the input graph $\Gamma$ is exactly the length 2 path from 1 to 3, with edge (23) excluded.  $F_f$ will count all induced copies of $P_2$ in the graph.   We can compute $h_e$ to be
$$h_e=2\left(p(p-1)\sqrt{p(1-p)}\right)+\sqrt{p(1-p)}p^2=p^{3/2}(1-p)^{1/2}\left(3p-2\right)$$
So $h_e\neq 0$ and $F$ is edge dominated so long as $p\neq \frac{2}{3}$.  Note this condition is quite logical, as $\frac23$ is the edge density of $P_2$, and intuitively it is at this point that observing an edge in our random graph gives us the least
information about how many induced copies of $P_2$ we should expect.
\end{example}

In general, these above examples extend to the case of all homomorphic or induced subgraph counts.  In particular, if $f$ checks for noninduced copies of a fixed grah $H$, then $F_f$ will always be edge dominated and obey the characteristic function bounds of Theorems \ref{subgraphmain1}, \ref{subgraphmain2}, and \ref{subgraphmain3} and the Central Limit Theorem of Theorem \ref{subgraphCLT}.  Meanwhile if $f$ counts induced copies of $H$, then $F_f$ will still be edge dominated \emph{so long as} $p\neq \frac{|E(H)|}{{k\choose 2}}$, that is $p$
is not exactly the edge density of $H$.

\subsection{Characteristic Function Bounds for Subgraph Counts and an Application}\label{subgraphmains}

Our first main result will be showing that the characteristic function of a function/random variable defined by applying an edge dominated graph statistic $F$ to $G(n,p)$ is close to that of the Gaussian. 
%

\begin{thm}\label{subgraphmain1}
Let $F:{[n] \choose 2}\to \mathbb{R}$ be an edge dominated graph statistic defined from $f:{[k]\choose 2}\to \mathbb{R}$ be as in definition \ref{graphstat}.  Let $Z:=\frac{F-\E F}{\sigma}$, then
$$\left|\varphi_Z(t)-e^{-\frac{t^2}{2}}\right|=O\left(\frac{t^3e^{-\frac{t^2}{3}}}{n}+\frac{t}{\sqrt{n}}\right)$$
\end{thm}
This result is always true, but useless for $t>>\sqrt{n}$.  To address the situation as $t$ grows larger we prove the following result.
\begin{thm}\label{subgraphmain2}
Fix $\epsilon>0$.  For $n^{\epsilon}<t\le n^{\frac12 +\frac \epsilon 4}$
$$|\varphi_Z(t)|\le O\left(\frac{1}{\sqrt{n}t^{1-\epsilon}}\right)$$
\end{thm}
Lastly we have one more case which covers yet more values of $t$.
\begin{thm}\label{subgraphmain3}
Fix $\epsilon>0$.  For $n^{\frac12+\epsilon}\le t\le n^{1-\epsilon}$ we have that
$$|\varphi_Z(t)|\le O\left(\frac{1}{tn^{1-\epsilon}}\right)$$
\end{thm}
We then show an application of all of these characteristic function bounds in the form of a quantitative central limit theorem for subgraph counts by the use of the Esseen Smoothing Lemma.  We restate an appropriate version
of the smoothing result (Lemma 2 of Chapter 16 in Feller \cite{Feller2} following a result of A.C. Berry).
\begin{lem}\label{smoothing}
Assume $Z$ has $\mathbb{E}[Z]=0$ and characteristic function $\varphi_Z(t)$.  Then if we let $\N(x):=\frac{1}{\sqrt{2\pi}}e^{-x^2/2}$, the density of the normal, and $\varphi:=e^{-t^2/2}$ be the characteristic function of the normal.  Finally let $\mathcal{Z}$ be the cumulative distribution function of $Z$ and $\mathfrak{N}$ the c.d.f. of the standard unit normal.  Then for any $x$ and $T>0$
\begin{align}
\left|\mathcal{Z}(x)-\mathfrak{N}(x)\right|\le \frac{1}{\pi}\int_{-T}^T\left|\frac{\varphi_Z(t)-\varphi(t)}{t}\right|dt+\frac{24}{\pi \sqrt{2\pi}T}
\end{align}
\end{lem}
We can now easily obtain the following quantitative central limit theorem for subgraph count like functions.
\begin{thm} \label{subgraphCLT}
Assume $F:{[n] \choose 2}\to \mathbb{R}$,  a graph statistic defined from $f:{[k]\choose 2}\to \mathbb{R}$, is edge dominated and $Z=\frac{F-\mu}{\sigma}$.   Then we have that for any $a<b$ fixed constants and $\epsilon>0$
$$\Pr(Z\in (a,b))=\frac{1}{\sqrt{2\pi}}\int_a^b e^{-x^2/2}dx+O_\epsilon\left(\frac{1}{n^{\frac12-\epsilon}}\right)$$
\end{thm}
\begin{proof}
Fix $T=\sqrt{n}$.  For all $t\le \sqrt{n}$ we can apply either Theorem \ref{subgraphmain1} or \ref{subgraphmain2} to bound
\begin{align*}
\frac{1}{\pi}\int_{-T}^T\left|\frac{\varphi_Z(t)-\varphi(t)}{t}\right|dt&\le 2\int_{0}^{n^\epsilon} \frac{1}{t}O\left(\frac{t^3e^{-\frac{t^2}{3}}}{n}+\frac{t}{\sqrt{n}}\right)dt+2\int_{n^{\epsilon}}^{\sqrt{n}} \frac{1}{t}O\left(\frac{1}{\sqrt{n}t^{1-\epsilon}}\right)dt\\
&=O_\epsilon\left(\frac{1}{n^{\frac{1}{2}-\epsilon}}\right)
\end{align*}
The result now follows immediately from Lemma \ref{smoothing}
\end{proof}
\subsection{Properties of Graph Statistics}
In this subsection we compute the Fourier Coefficients, variance and spectral concentration of $F$, where $F$ is a graph statistic defined from $f$ as in Definition \ref{graphstat}.  
Fix a set $T\subset {[n]\choose 2}$.  Note that a map $\psi\in \mathcal{S}_G$ such that $\supp(T)\subset \psi(D)$ can be 
determined as follows:  Pick an injection $\phi:\supp(T)\hookrightarrow D$, and for $v\in \phi(\supp(T))$ set $\psi(v)=\phi^{-1}(v)$.  We can then extend $\psi$ to a map on all of $[k]$ it by specifying the image of $\psi$ on $\phi(T)^c$ arbitrarily.
An extension can be picked in $(n-|\supp(T)|)\downarrow_{k-|\supp(T)|}$ ways.  So we have that 
\begin{align}\label{subgraphfouriercoeff}
\begin{split}
\hat F(T)&=\sum_{\substack{\psi\in \mathcal{S}_G\\ \supp(T)\subset \psi(D)}} \hat f(\psi^{-1}(T))=\sum_{\phi:\supp(T)\hookrightarrow D}\sum_{\psi:[k]-\varphi(T)\hookrightarrow [n]} \hat f(\psi^{-1}(T))\\
&=(n-|\supp(T)|)\downarrow_{k-|\supp(T)|}\sum_{\varphi:\supp(T)\hookrightarrow D}\hat f(\psi^{-1}(T))\\
&=(n-|\supp(T)|)\downarrow_{k-|\supp(T)|}h_T
\end{split}
\end{align}
Furthermore, this shows us that $\hat F(T)=\Theta(n^{k-|\supp(T)|})$, so long as we have that $h_T\neq 0$.
It is of particular importance whether $h_e=0$ were $e$ is a set consisting of a single edge.  
Using these estimates and Parseval's \ref{parseval} we can compute the variance of $F$ to be
\begin{align}\label{subgraphvar}
\begin{split}
\sigma^2:=Var(F)&={n\choose 2}h_e^2\left((n-2)\downarrow_{k-2}\right)^2+\sum_{i=3}^k\left((n-i)\downarrow_{k-i}\right)^2\sum_{|\supp(T)|=i} h_T^2\\
&={n\choose 2}h_e^2n^{2k-2}+O\left(\sum_{i=3}^k\left((n-i)\downarrow_{k-i}\right)^2\left[{n\choose i}\max_{|\supp(T)=i|}h_T^2\right]\right)\\
&=\sum{n\choose 2}h_e^2n^{2k-2}+O(n^{2k-3})
\end{split}
\end{align}
So we see that if $h_e\neq 0$, that is $f$ is edge dominated, then $\sigma^2-W^1(F)=O(n^{2k-3})$.  In fact we have shown that more is true, and that for any $j\ge 1$ we have that $W^j(F)/\sigma^2 = O(n^{-j+1})$

\subsection{Small values of $t$}

The goal of this subsection is to prove Theorem \ref{subgraphmain1}, which we restate.
\begin{repthm}{subgraphmain1}
Let $F:{[n] \choose 2}\to \mathbb{R}$ be a graph statistic defined from $f:{[k]\choose 2}\to \mathbb{R}$ be as in Definition \ref{graphstat}.  Assume $F$ is edge dominated, that is $h_e=\sum_{e\in {[k]\choose 2}} \hat f(e)\neq 0$.   Let $Z:=\frac{F-\E F}{\sigma}$, then
$$\left|\varphi_Z(t)-e^{-\frac{t^2}{2}}\right|=O\left(\frac{t^3e^{-\frac{t^2}{3}}}{n}+\frac{t}{\sqrt{n}}\right)$$
\end{repthm}

%

%

 Theorem \ref{subgraphmain1} concerns $Z$, a normalized form of $F$ with mean 0 and variance 1.  
We can decompose $Z$ into two parts, as we did in the triangle case, the dominant weight one part $X$, and a smaller term $Y$ corresponding to Fourier coefficients of weight $\ge 2$.  Let $Q:=\sqrt{\frac{1}{{n \choose 2}}}$ and 
\begin{align*}
X:=\sum_{e\in {[n]\choose 2}} Q \chi_e&&Y:=\sum_{e\in {[n]\choose 2}} (\hat Z(e)-Q)\chi_e+\sum_{|S|\ge 2} \hat Z(S) \chi_S
\end{align*}
First we examine $X$.  It is the mean 0 variance 1 sum of independent  random variables, and so by Berry-Esseen (see Petrov \cite{Petrov}, Chapter V lemma 1) we know that if 
$$L_n:={n\choose 2}\E[|Q\chi_e|^3]=\frac{p^2+(1-p)^2}{\sqrt{{n\choose 2}p(1-p)}}=\Theta_p\left(1/n\right)$$
then for $t\le \frac{1}{4L_n}$ we have that
\begin{equation}\label{eqBerryEsseen}
\left|\E[e^{itX}]-e^{-\frac{t^2}{2}}\right| \le 16L_n|t|^3e^{\frac{-t^2}{3}}
\end{equation}
Next we examine $Y$.  It is best considered as an error term, and we will show that $\E|Y|$ is small.
We know from prior calculations \ref{subgraphfouriercoeff} and \ref{subgraphvar}
$$\sum_{|S|\ge 2} \hat Z^2(S)=O\left(\frac{kn^{2k-3}}{\sigma^2}\right)=O_k\left(\frac{1}{n}\right)$$
Further we can estimate
$${n\choose 2}\sigma^2\hat Z^2(e)-\sigma^2= {n\choose 2}\hat F^2(e)-\sigma^2=O(n^{2k-3})\implies \hat Z^2(e)-\frac{1}{{n\choose 2}}=O(n^{-3})$$
Therefore using the fact that $(x-y)=(x^2-y^2)/(x+y)$ coupled with the observation that $\hat Z(e)+Q=\Theta\left(\frac{1}{n}\right)$, we find that
$$|\hat Z(e)-Q|\le \left|\frac{ \hat Z^2(e)-\frac{1}{{n\choose 2}}}{\hat Z(e)+Q}\right|=O\left(\frac{1}{n^2}\right)$$
So as a result we can conclude that $var(Y)=O(1/n)$ and so $\E[|Y|]=O(\frac{1}{\sqrt n})$.  Now we are ready for our characteristic function bound for $Z$.  If $|t|\le \frac{1}{4L_n}=\Theta_p(n)$ then combining the above with equation \ref{eqBerryEsseen} yields.

\begin{align*}
\left|\varphi_Z(t)-e^{-\frac{t^2}{2}}\right|&=\left|\E\left[e^{itZ}\right]-e^{-\frac{t^2}{2}}\right|=\left|\E\left[e^{it(X+Y)}\right]-e^{-\frac{t^2}{2}}\right|\le \left|\E\left[e^{itX}\right]-e^{-\frac{t^2}{2}}\right|+\left|\E\left[e^{itX+Y}\right]-\E e^{itX}\right|\\
&\le 16L_n|t|^3e^{\frac{-t^2}{3}}+\E|tY|=O\left(\frac{t^3e^{-\frac{t^2}{3}}}{n}+\frac{t}{\sqrt{n}}\right)
\end{align*}
But this is exactly the statement of Theorem \ref{subgraphmain1}

\subsection{Bounds for slightly larger $t$}
The goal for this subsection is to prove
\begin{repthm}{subgraphmain2}
Fix $\epsilon>0$.  For $n^{\epsilon}<t\le n^{\frac12 +\frac \epsilon 4}$
$$|\varphi_Z(t)|\le O\left(\frac{1}{\sqrt{n}t^{-1+\epsilon}}\right)$$
\end{repthm}
To prove this statement we will first need the following claims:

\begin{claim}\label{goodgraph}
Fix $\epsilon>0$.  For all sufficiently large $n$, we have that for any $\alpha\in (n^{-1+\epsilon},1)$
there exists a set of edges $H\subset {[n] \choose 2}$  with $|H|\ge \frac{\alpha n}{2}$ such that
$$\sum_{\substack{S\in H\\|S|\ge 2}} n^{2k-2|\supp(S)|}\le C \alpha^2 n^{2k-3}$$
Where $C$ is a fixed constant depending only on $f$.
\end{claim}

\begin{claim}\label{GoodEvent}
Let $A$ be the event (over the space of revelations $\beta \in \{0,1\}^{H^c}$) that for \emph{every} edge $e\in H$ we have that
$$|\widehat{Z_\beta}(e)-\hat{Z}(e)|<\frac{1}{n^{1.4}}$$
Let $\lambda:=\min(p,1-p)$.  Then $\Pr(A)\ge 1-n^2\lambda^2e^{-\Omega\left(\lambda n^{\frac{0.1}{k^2}}\right)}$.
\end{claim}

\begin{claim}\label{goodtimes}
Let $B$  be the event (over the space of revelations $\beta \in \{0,1\}^{H^c}$) that for \emph{every} set $S\subset E$ with $|S|\ge 2$
$$|\widehat{Z_\beta}(S)|\le C n^{k-s}$$
where $C$ is the fixed constant  $C:=h_*2^{{k\choose 2}}+1$ and $s=|\supp(S)|$.
Let $\lambda:=\min(p,1-p)$.  Then $\Pr(B)\ge 1-O\left(n^k e^{-\Omega\left(n^{\frac{2}{k^2}}\right)}\right)$
\end{claim}

\begin{claim}\label {goodcase}
Assume $\beta \in A\cap B$.  Then for any $\alpha\in (n^{-1+\epsilon}, 1)$ and $t=o(n)$
$$\E_{x\in 2^{H}}[ e^{itZ_\beta}]\le \exp\left(-\frac{\alpha t^2}{8\pi^2}\right)+O\left(|t|\alpha n^{k-\frac{3}{2}}\right)$$
\end{claim}

\begin{claim}\label{middlet}
For $\alpha \in (n^{-1+\epsilon},1)$
 we have that 
$$|\varphi_Z(t)|<\exp\left(-\frac{\alpha t^2}{8\pi^2}\right)+O\left(\alpha |t| n^{k-\frac{3}{2}}+n^k e^{-\Omega\left(n^{\frac{2}{k^2}}\right)}+n^2e^{-\Omega\left(\lambda n^{\frac{0.1}{k^2}}\right)}\right)$$
\end{claim}

Theorem \ref{subgraphmain2} now follows simply by making a good choice of $\alpha$.
\begin{proof}[Proof of Theorem \ref{subgraphmain2}]
We can now fix $\alpha$ to be of size $t^{-2+\epsilon}$, which  is feasible for the hypothesis of Claim \ref{goodgraph} so long  as we assure that $n^\epsilon<t<n^{\frac12+\frac{\epsilon}{4}}$, and so $\alpha>n^{-1+\epsilon/2}$.
Plugging this choice of $\alpha$ into Claim \ref{middlet} completes the proof.
\end{proof}

\subsubsection{Proof of Claims}
\begin{proof}[Proof Of Claim \ref{goodgraph}]
Fix $\ell=\lfloor \alpha n \rfloor$.  Let $H$ be the subgraph given by taking the union of $\lfloor\frac{n}{\ell}\rfloor$ disjoint cliques of size $\ell$, and the remaining vertices with no edges.
The number of edges in $H$ is 
$${\ell \choose 2}\lfloor \frac{n}{\ell}\rfloor\ge \frac{n(\ell-1)}{2}-{\ell\choose 2}\ge \frac{\alpha n^2}{2}-\ell^2-\frac{n}{2}$$

Meanwhile the number of subgraphs of $H$ with support of size $|\supp(S)|=i$  is upper bounded by
$$\lfloor\frac{n}{\ell}\rfloor\left({\ell \choose i}2^{{i\choose 2}}\right)\le n\ell^{i-1}2^{i^2}=\alpha^{i-1}n^i(2^{i^2}+O(\frac{1}{\ell}))$$
So we can see that the number of edges is at least $\frac{\alpha n}{2}$ for $n$ sufficiently large.  Further we can compute that
\begin{align*}
\sum_{\substack{S\in H\\|S|\ge 2}} n^{2k-2|\supp(S)|}&\le \sum_{i=3}^k\sum_{\substack{S\in H\\|\supp(S)|=i}} n^{2k-2i}\le \sum_{i=3}^{k} 
(2^{i^2}+O(\frac{1}{\ell}))\alpha^{i-1}n^in^{2k-2i}\\
&\le (k+O(\frac{1}{\ell}))\alpha^2n^{2k-i}=O(\alpha^2n^{2k-3})
\end{align*}
Where the last inequality is justified by the assumption that $\ell \ge h(n)$ where $n\to \infty$.
\end{proof}

We prove Claim \ref{GoodEvent} by noting that the formula for $\widehat{Z_\beta}(S)$ (a coefficient in the polynomial $Z_\beta$) is \emph{itself} a low degree polynomial, and therefore may be shown to have tight concentration by hypercontractivity. 

\begin{proof}[Proof Of Claim \ref{GoodEvent}]

Recall that 
$$\widehat{Z_\beta}(e)=\sum_{T\subset H^c} \hat Z(e\cup T)\chi_{T}(\beta)$$

So $\widehat{Z_\beta}(e):\{0,1\}^{H^c}\to\mathbb{R}$ is a polynomial (in the functions $\chi_e$), and we can began by estimating its coefficients.
First we see that 
$$\E[\widehat{Z_\beta}(e)] =\widehat{\widehat{Z}_\beta(e)}(\varnothing)=\hat{Z}(e)$$

Also for any $T\subset \{0,1\}^{H^c}$ we know that $\hat{Z}(e\cup T)\neq 0$ only if $|\supp(e\cup T)|\le k$.  So we can compute:

\begin{align*}
Var_\beta( \widehat{Z_\beta}(e))&=\sum_{\substack{T\subset H^c\\T\neq \varnothing}} \hat{Z}(e\cup T)^2=\sum_{i=3}^k \sum_{\substack{T\subset H^c\\|\supp(T\cup e)|=i}} \hat Z(e\cup T)^2\\
&\le \sum_{i=3}^k\sum_{|\supp(T\cup e)|=i} \hat Z(e\cup T)^2\le \sum_{i=3}^k {n-2 \choose i-2}\frac{h_*^2n^{2(k-i)}}{\sigma^2}\le k h_*^2\frac{n^{2k-5}}{\sigma^2}\\
&=O\left(\frac{1}{n^3}\right)
\end{align*}

Since $\widehat{Z_\beta}(e)$ has degree less than ${k\choose 2}$, an application of Theorem \ref{Hypercontractivity} gives us that for any $e\in H$ if we set $\lambda=\min(p,1-p)$ then
$$\Pr\left[\left|\widehat{Z_\beta}(e)-\hat Z(e)\right|\ge \frac{1}{n^{1.4}}\right]<\lambda^2\exp\left(-\Omega\left(\frac{\lambda n^{\frac{0.1}{k^2}}}{e}\right)\right)$$
Applying a union bound over all edges in $H$ completes the proof.
\end{proof}

\begin{proof}[Proof of Claim \ref{goodtimes}]

Again we use the decomposition 
$$\widehat{Z_\beta}(S)=\sum_{T\subset H^c} \hat Z(S\cup T)\chi_{T}(\beta)$$

So $\widehat{Z_\beta}(S):\{0,1\}^{H^c}\to\mathbb{R}$ is a polynomial (in the functions $\chi_e$), and we can began by estimating its coefficients.
First we see that 
$$\E[\widehat{Z_\beta}(S)] =\widehat{\widehat{Z}_\beta(S)}(\varnothing)=\hat{Z}(S)$$

Assume $|\supp(S)|=s$.  For any $T\subset \{0,1\}^{H^c}$ we know that $\hat{Z}(S\cup T)\neq 0$ iff $|\supp(S\cup T)|\le k$.  There are at most $2^{{\ell\choose 2}}(n-s)\downarrow_{\ell-s}\le 2^{k^2}n^{\ell-s}$ choices of $T$ such that $|\supp(S\cup T)|=\ell$.  And further for each of these choices we know that $\hat Z(S\cup T)\le h_*n^{k-\ell}$.  
Let $g:=\sum_{\substack{T\subset H^c\\|\supp(S\cup T)|>s}} \hat Z(S\cup T)\chi_T(\beta)$
So we can compute that
\begin{align*}
Var(g)&\le \sum_{\ell=s+1}^k \sum_{|\supp(S\cup T)=\ell} \left(\hat{Z}(S\cup T)\right)^2 \le \sum_{\ell=s+1}^{k} 2^{k^2}n^{\ell-s} (h_*)^2n^{2k-2\ell}\\
&\le k2^{k^2}(h_*)^2 n^{2k-2s-1}
\end{align*}
Further we can see that $g$ is a polynomial of degree at most $2^{k\choose 2}$, and so by Hypercontractivity \ref{Hypercontractivity} we see that 
\begin{align*}
\Pr\left[|g|\ge n^{k-s}\right]&=\Pr\left[g\ge \frac{1}{ \sqrt{k2^{k^2}(h_*)^2} }\sqrt{n}\|g\|_2\right]\le \lambda^{{k\choose 2}}\exp\left(-\frac{{k\choose 2}}{2e}\lambda \left(\frac{t}{k2^{k^2}(h_*)^2}\right)^{\frac{2}{{k\choose 2}}}\right)\\
&=O(e^{-\Omega(n^{\frac{2}{k^2}})})
\end{align*}

If $|g|<n^{k-s}$ then we can conclude that 
\begin{align*}
\hat Z(S)=\sum_{|\supp(S\cup T)|=s} \hat Z(S\cup T)\chi_T(\beta)+g(\beta)\le 2^{{s\choose 2}} h_*n^{k-s}+n^{k-s}
\end{align*}
So for any $S\subset H$ we find that 
$|\widehat{Z_\beta}(S)|\le C n^{2k-2s}$  with probability at least $1-O(e^{-\Omega(n^{\frac{2}{k^2}})})$.  Taking a union bound over all such $S$ finishes the proof.
\end{proof}

\begin{proof}[Proof of Claim \ref{goodcase}]
Assume that $\beta\in A\cap B$.  Let $X$ and $Y$ be
\begin{align*}
X:=\sum_{e\in {[n]\choose 2}} \widehat{Z_\beta}(e) \chi_e&&Y:=\sum_{|S|\ge 2} \widehat{Z_\beta}(S) \chi_S
\end{align*}
then $Z_\beta=X+Y$, where $X$ is an independent sum, and $Y$ is likely small, so we will be able to obtain bounds similar to our previous ones.
Let $Q=\sqrt{\frac{1}{{n\choose 2}}}=(1+O(\frac{1}{n}))\hat Z(e)$, and that $Q\approx \frac{\sqrt{2}}{n}$.  Because of our assumption that $\beta\in A$ we have that
$$\frac{Q}{2}\le \hat{Z}(e)-n^{-1.4}\le \widehat{Z_\beta}(e) \le \hat{Z}(e)+n^{-1.4}\le \frac{3Q}{2}$$

Using our bound on $\widehat{Z_\beta}(e)$ and the fact that $t=o(1/Q)$ we may apply  Lemma $\ref{Bernoulli}$ to say that
$$|\E[e^{it\widehat{Z_\beta}(e)\chi_e}]|= 1-\frac{t^2Q^2}{2\pi^2}\le \exp\left(-\frac{t^2}{2n^2 \pi^2}\right)$$ 

So now we find that 
\begin{align*}
\E[e^{itX}]=\prod_{e\in H}\E[\exp\left(it\widehat{Z_\beta}(e)\chi_e\right)]\le \exp\left(-\sum_{e\in H}(t\widehat{Z_\beta(e)})^2\right)\le \exp\left(\sum_{e\in H}-\frac{t^2}{2\pi^2n^2}\right)=\exp\left(-\frac{\alpha n^2}{4}\cdot\frac{t^2}{2\pi^2n^2}\right)
\end{align*}

Next we turn our attention to $Y$.
We can use Cauchy Schwartz, the assumption that $\beta\in B$ and the fact that $H$ satisfies the conditions of Claim \ref{goodgraph} to bound
$$\E[|Y|]^2\le \E[|Y|^2]=\sum_{\substack{S\subset H\\|S|\ge 2}} \widehat{Z_\beta}^2(S)\le \sum_{\substack{S\subset H\\|S|\ge 2}} C^2 n^{2k-2|\supp(S)|}\le O(\alpha^2n^{2k-3})$$
Finally we combine all of these estimates to bound $\E_H[e^{itZ_\beta}]$ and finish the proof of Claim \ref{goodcase}

\begin{align*}
\left|\E_{\alpha \in 2^{H}} [e^{itZ_\beta(\alpha)}]\right|&=\left|\E_{\alpha}[e^{it(X+Y)}]\right|\le \left|\E[e^{itX}+|tY|]\right|\\
&\le\exp\left(-\frac{\alpha t^2}{8\pi^2}\right)+O\left(|t|\alpha n^{k-\frac{3}{2}}\right)
\end{align*}
\end{proof}
\begin{proof}[Proof of Claim \ref{middlet}]
Let $A$, and $B$ be as defined in Claims \ref{GoodEvent} and \ref{goodtimes}.  
We can break up $\{0,1\}^{H^c}$ into $A\cap B$ and $(A\cap B)^c$ and estimate

\begin{align*}
|\varphi_Z(t)|&:=\E_{(\alpha,\beta)\in 2^{n\choose 2}}[e^{itZ(\alpha,\beta)}]\le \E_{\beta\subset H^c}|\E_{\alpha\subset H}[e^{itZ_\beta(\alpha)}]|\le \Pr[(A\cap B)^c]+ \Pr[A\cap B]\E_{\beta \in (A\cap B)^c}\left|\E_{\alpha}[e^{itZ_\beta}]\right|
\end{align*}
Now combining Claims \ref{GoodEvent} and \ref{goodtimes} we find that

\begin{align*}
\Pr[(A\cap B)^c]+ Pr[A\cap B]\E_{\beta \in A}\left|\E_{\alpha}[e^{itZ_\beta}]\right|&\le  \exp\left(-\frac{\alpha t^2}{8\pi^2}\right)+O\left(\alpha |t| n^{k-\frac{3}{2}}+n^k e^{-\Omega\left(n^{\frac{2}{k^2}}\right)}+n^2e^{-\Omega\left(\lambda n^{\frac{0.1}{k^2}}\right)}\right)
\end{align*}

\end{proof}

\subsection{Middle values of $t$}
This subsection does not have a direct analog in the triangle case, as the tighter Cauchy-Schwarz bound given in \cite{JustinTriangles} may be used in that case.
The goal of this subsection is to prove
\begin{repthm}{subgraphmain3}
Fix $\epsilon>0$.  For $n^{\frac12+\epsilon}\le t\le n^{1-\epsilon}$ we have that
$$|\varphi_Z(t)|\le O\left(\frac{1}{tn^{1-\epsilon}}\right)$$
\end{repthm}

For $t\ge n^{\frac12+\epsilon}$ we use a different choice of $H$, the subgraph whose complement we reveal, from in the previous arguments.  Here we take $H$ to be a matching of size $\ell$.  Again let $\beta \in 2^{{H^c}}$ be a revelation of all of the edges in $H^c$ and look at

$$Z_\beta = Z(\x_H, \beta)=X_\beta+Y_\beta$$
where
\begin{align*}
X_\beta:=\sum_{e\in H} \widehat{Z_\beta}(e) \chi_e&&Y_\beta:=\sum_{\substack{S\subset H\\|S|\ge 2}} \widehat{Z_\beta}(S) \chi_S
\end{align*}
Because $H$ is a matching, any set $S\subset H$ has support of size $|\supp(S)|=2|S|$.  So assuming that we are again in the event $|A\cap B|$ (i.e. all of the Fourier coefficients are behaved nicely where $A$ and $B$ are as defined in Claims
\ref{GoodEvent} and \ref{goodtimes} respectively) we can compute that
\begin{align*}
\E[|Y_\beta|]^2&\le \E[|Y_\beta|^2]= \sum_{\substack{S\subset H\\|S|\ge 2}}\widehat{Z_\beta}^2(S)\le \sum_{i=2}^{2\ell} \sum_{\substack{S\subset H\\|S|=i}} \frac{C^2}{\sigma^2}n^{2k-4i}\\
&\le \sum_{i=2}^{2\ell} {\ell\choose i}\frac{C^2}{\sigma^2}n^{2k-4i}\le 2\ell^2\frac{C^2}{\sigma^2}n^{2k-8}
\end{align*}
So we have $\E[|Y_\beta|] = O(\ell n^{-3})$.
Meanwhile so long as $t\widehat{Z_\beta}(e)<\sqrt{p(1-p)}\pi $ we can use Lemma \ref{Bernoulli} to compute
\begin{align*}
\E[e^{itX_\beta}]&=\prod_{e\in H}\E[\exp\left(it\widehat{Z_\beta}(e)\chi_e\right)]\le \exp\left(-\sum_{e\in H}(t\widehat{Z_\beta(e)})^2\right)\\
&\le
\exp\left(\ell \frac{t^2}{2n^2}\right)
\end{align*}
So for $n^{\frac{1}{2}+\epsilon}<t<n^{1-\epsilon}$ we can choose $\ell =\lfloor \frac{n^{2+\epsilon}}{t^2}\rfloor\in [n/2]$ (the size bound verifying that there does exist a matching of size $\ell$) and then we find
$\E[e^{itX_\beta}]\le \exp(-n^\epsilon)$.
So then in total we have that if $\beta \in A\cap B$ then
$$\E_{H}[e^{itZ_\beta}]=\E[e^{it(X_\beta+Y_\beta)}]\le \E[e^{itX_\beta}+t|Y_\beta|]\le e^{-n^{\epsilon}}+O\left(\frac{t\ell}{n^3}\right)=e^{-n^{\epsilon}}+O\left(\frac{1}{tn^{1-\epsilon}}\right)$$
Also, arguing as we did in Claim \ref{middlet}, we can use Claims \ref{GoodEvent} and \ref{goodtimes} to show that $\Pr[\beta\in (A\cap B)^c]\le O\left(\frac{1}{tn^{1-\epsilon}}\right)$.  So we can now conclude that
$$\E[e^{itZ}]=\E_{\beta}\E_{H} e^{itZ}\le O\left(\frac{1}{tn^{1-\epsilon}}\right)$$
Concluding the proof.

\bibliographystyle{alpha}
\bibliography{TryNBib}
\end{document}